\documentclass[final]{siamltex}

\usepackage[parfill]{parskip}    
\usepackage{amsmath,amssymb,latexsym,enumerate,mathrsfs}
\usepackage{hyperref}
\usepackage{array}
\usepackage{times}
\usepackage{subfig}
\usepackage[]{mdframed}
\usepackage{epstopdf}
\usepackage{subfig}
\usepackage{graphicx}
\usepackage{hyperref} 
\usepackage{multicol}
\usepackage{framed}
\usepackage{moreverb}
\usepackage{marvosym}
\usepackage{color,soul}
\definecolor{lightblue}{rgb}{.90,.95,1}
\sethlcolor{yellow}

\DeclareMathOperator*{\argmin}{arg\,min}

\newtheorem{remark}{Remark}

\newtheorem{assumption}{\bf Assumption}

\title{A new perspective on the learning dynamics for a class of learning problems via averaged gradient systems coupled with diffusion-transmutation processes}

\author{Getachew K. Befekadu}

\begin{document}
\maketitle

\renewcommand{\thefootnote}{\arabic{footnote}}

\begin{abstract}
In the first part of this paper, we consider a family of continuous-time dynamical systems coupled with diffusion-transmutation processes. Under certain conditions, such randomly perturbed dynamical systems can be interpreted as an averaged dynamical system, whose weighting coefficients, that depend on the state trajectory of the underlying averaged system, are assumed to be strictly positive with sum unity. Here, we provide a large deviation result for the corresponding family of processes, i.e., a variational problem formulation modeling the most likely sample path leading to certain noise-induced rare-events. This remarkably allows us to provide a computational algorithm for solving the corresponding variational problem. In the second part of the paper, we use some of the insights from the first part and provide a new perspective on the learning dynamics for a class of learning problems, whose averaged gradient dynamical systems, from continuous-time perspective, are guided by a set of subsampled datasets that are obtained from the original dataset via bootstrapping or other related resampling-based techniques. Finally, we present some numerical results for a typical nonlinear regression problem, where the corresponding averaged gradient system is interpreted as random walks on a graph, whose outgoing edges are uniformly chosen at random.
\end{abstract}

\section{Introduction} \label{S1}
The main objective of this paper is to provide a new perspective on the learning dynamics for a class of learning problems of point estimations for modeling of high-dimensional nonlinear functions or dynamical systems. In particular, the framework can be viewed as a general extension of improving or enhancing the learning dynamics, by apportioning dynamically the averaging weighting coefficients, for a family of empirical risk minimization-based learning problems, where the corresponding gradient systems, from continuous-time perspective, are guided by a set of subsampled datasets obtained from the original dataset. To be more specific, there are two conceptually related topics discussed in some details. (i) In the first part, we specifically consider a family of continuous-time dynamical systems coupled with diffusion-transmutation processes. Under certain conditions, we can interpret such randomly perturbed dynamical systems as an averaged dynamical system, whose weighting coefficients, that depend on the state trajectory of the underlying averaged system, are strictly positive with sum unity. Here, we also provide a large deviation result for the corresponding family of processes, i.e., a variational problem formulation modeling the most likely sample path leading to certain noise-induced rare-events. This further allows us to provide a computational algorithm, i.e., a rare-even algorithm, for solving numerically the corresponding variational problem. (ii) In the second part, we use some of the insights from the first part and we further provide a new perspective to understand the learning dynamics for a class of learning problems, whose averaged gradient systems, from continuous-time perspective, are guided by a set of subsampled datasets obtained from the original dataset via bootstrapping or other related resampling-based techniques, as an instance of {\it divide-and-conquer} paradigm dealing with large datasets. 

This paper is organized as follows. In Section~\ref{S2}, we first consider a family of continuous-time dynamical systems coupled with diffusion-transmutation processes. In this section, we also discuss conditions under which such randomly perturbed dynamical systems can be interpreted as an averaged dynamical system, whose weighting coefficients are strictly positive, with sum unity, that depend on the state trajectory of the underlying averaged system, where such a connection further allows us to provide a large deviation result for the corresponding family of processes. In Section~\ref{S3}, based on the insights from Section~\ref{S2}, we present our main results, where we specifically consider an empirical risk minimization-based learning problem, whose gradient dynamical systems is guided by a set of subsampled datasets obtained from the original dataset. This section also provides a rare-event algorithm for solving the corresponding variational problem. In Section~\ref{S4}, we present some numerical results for a typical case of nonlinear regression problem, and  Section~\ref{S5} contains some concluding remarks.

\section{Large deviations, action functionals and optimal sample path}\label{S2} In this section, we present some background and preliminary results that will later be useful for our main results. Here, the mathematical argument is sufficient to provide a new perspective on the learning dynamics for a class of learning problems, whose averaged gradient dynamical systems are guided by a set of subsampled datasets obtained from the original dataset.

Consider the following Markov process $\left(X_t^{\epsilon}, \nu_k^{\epsilon}, \mathbb{P}_{x,k}\right)$ such that
\begin{align}
d X_t^{\epsilon} &= f_{ \nu_k^{\epsilon}}(X_t^{\epsilon}) dt + \sqrt{\epsilon} I_p d W_t, \quad X_0^{\epsilon} = x_0, \quad \label{Eq2.1}\\ ~\notag\\
 \mathbb{P} \left \{\nu_{t + \Delta}^{\epsilon} = l \biggl \vert X_t^{\epsilon} = x, \, \nu_{t}^{\epsilon} = k \right\} & = \frac{p_{kl}(x)}{\epsilon} \Delta + o({\Delta}) \quad \text{as} \quad\Delta \downarrow 0, \label{Eq2.2}\\
& \quad k, l \in \left\{1,\,2,\ldots, K\right\} \quad \text{and} \quad x \in \mathbb{R}^p, \notag
\end{align}
where $\epsilon \ll 1$ is a small positive parameter, $I_p$ is a $p \times p$ identity matrix, $W_t$ is a $p$-dimensional standard Wiener process, and the intensities for the second component of the Markov process are assumed to be strictly positive, i.e., $p_{kl}(x) > 0$, for $k,\l \in \{1,\,2,\ldots, K\}$, when $x \in \mathbb{R}^p$ is fixed. Moreover, the intensity for the additive noise term in Equation~\eqref{Eq2.1} depends on the small parameter $\epsilon \ll 1$, while the frequencies for the fast random switching $\nu_{\cdot}^{\epsilon}$ in Equation~\eqref{Eq2.2} are determined by $\epsilon^{-1}$. 

Throughout the paper, we will assume the following condition which is a counterpart assumption made on asymptotic problems concerning random differential equations with small noise intensities (e.g., see \cite{r1}, \cite{r2}, \cite{r3} or \cite{r4} for additional discussions).
\begin{assumption}\label{ASM1}
Let $D \subset \mathbb{R}^p$ be a given bounded open domain, with smooth boundary $\partial D$. Moreover, assume that there exists a compact set $\,\Gamma \subset D$ such that any two points of $\,\Gamma$ are equivalent and that the {\it $\omega$-limit} sets of trajectories for the family of dynamical systems
\begin{align}
 \dot{x}_k(t) = f_k(x_k), \quad  k = 1, 2, \ldots, K, \label{Eq2.3}
\end{align}
that start from $D \cup \partial D$ always belong to $\Gamma$.
\end{assumption}

Consider next the following continuous-time averaged dynamical system
\begin{align}
 \dot{x}(t) &= \sum\nolimits_{k=1}^K \pi_k(x) f_k(x), \quad x(0) = x_0,  \quad  \left(\text{with} ~~ \dot{x}(t) = \tfrac{d x(t)}{dt} \right), \label{Eq2.4}\\
                &\triangleq \bar{f}(x), \notag
\end{align}
where $\pi_k(x) > 0$ for all $k \in \{1,\,2,\ldots, K\}$, with $\sum\nolimits_{k=1}^K \pi_k(x) = 1$, and $x \in \mathbb{R}^p$. 

Then, we state the following two propositions that provide conditions under which the above averaged dynamical system in Equation~\eqref{Eq2.4}, in the low-noise limits as $\epsilon \to 0$, can be interpreted as a family of randomly perturbed dynamical systems coupled with diffusion-transmutation processes of Equations~\eqref{Eq2.1} and \eqref{Eq2.2}.

\begin{proposition} \label{P1}
Suppose that the $K \times K$ stochastic matrix $P(x) = \left(p_{kl}(x) \right)_{k,l}$, for $k, l \in \left\{1,\,2,\ldots, K\right\}$, in Equation~\eqref{Eq2.2} is irreducible and aperiodic, whenever $x \in \mathbb{R}^p$ is fixed. Then, there exits a unique strictly positive stationary (or invariant) distribution $\pi(x)$, i.e., a unique left eigenvector of $P(x)$, such that
\begin{align}
 \pi(x) = \pi(x)P(x), \label{Eq2.5}
\end{align}
where $\pi(x)= \left(\pi_1(x),\pi_2(x), \dots, \pi_K(x) \right)$, with $\pi_k(x) > 0$, for $k =1,\,2,\ldots, K$, and ~ $\sum\nolimits_{k=1}^K \pi_k(x) = 1$.
\end{proposition}

\begin{proof}
The proof is a direct application of the fundamental theorem of Markov process (e.g., see \cite{r5} or \cite{r6}).
\end{proof}

\begin{proposition} \label{P2}
Suppose that Proposition~\ref{P1} holds true. Then, the first component of the Markov process $\left(X_t^{\epsilon}, \nu_k^{\epsilon}, \mathbb{P}_{x,k}\right)$ corresponding to the system in Equations~\eqref{Eq2.1} and \eqref{Eq2.2} approaches the state trajectory of the averaged dynamical system in Equation~\eqref{Eq2.4}, in the low-noise limits as $\epsilon \to 0$, with the same initial condition $X_0^{\epsilon} = x(0) = x_0$. 
\end{proposition}

\begin{proof}
The proof follows from the fact that the second component of the Markov process $\left(X_t^{\epsilon}, \nu_k^{\epsilon}, \mathbb{P}_{x,k}\right)$ in Equations~\eqref{Eq2.1} and \eqref{Eq2.2} approaches to a unique stationary distribution $\pi(x)$ of Markov chain with intensities $p_{kl}(x) > 0$, for $k, l \in \left\{1,\,2,\ldots, K\right\}$, whenever $x \in \mathbb{R}^p$ is fixed. As a result of this, the first component $X_t^{\epsilon}$ can be viewed as a random perturbation of the averaged dynamical system of Equation~\eqref{Eq2.4}, with small noise term, i.e.,
\begin{align*}
 d X_t^{\epsilon} = \sum\nolimits_{k=1}^K \pi_k(X_t^{\epsilon}) f_k(X_t^{\epsilon}) dt + \sqrt{\epsilon} I_p d W_t, \quad X_0^{\epsilon} = x_0,
 \end{align*}
 which provides a rationale for the proposition and we omit the details (see \cite{r7} and \cite{r8} for similar discussions). 
 \end{proof}

Here, it is worth remarking that, for small $\epsilon \ll 1$, the second component of the Markov process in Equations~\eqref{Eq2.1} and \eqref{Eq2.2} makes many transitions approximately with frequencies $\epsilon^{-1} \gg 1$, while the first component makes a small displacement from the starting point $x \in \mathbb{R}^p$. On the other hand, the stochastic matrix $P(x) = \left(p_{kl}(x) \right)_{k,l}$, for $k, l \in \left\{1,\,2,\ldots, K\right\}$, is assumed to be irreducible and aperiodic, for a fixed $x \in \mathbb{R}^p$. As a result of this, the corresponding distribution for the second component will approach to a unique stationary (or invariant) distribution $\pi(x)$, with intensities $p_{kl}(x) > 0$, for $k, l \in \left\{1,\,2,\ldots, K\right\}$, whenever $x \in \mathbb{R}^p$ is fixed. 

Then, noting Assumption~\ref{ASM1}, our main interest is to provide a large deviation result, i.e., characterizing the logarithmic order for the probabilities of the sample path deviations of $X_t^{\epsilon}$ from $x(t)$, i.e, the state trajectory of the averaged dynamical system in Equation~\eqref{Eq2.4}, in the interval $[0, T]$, due to the small diffusion term as well as the deviations from the unique stationary distribution, in the low-noise limits as $\epsilon \to 0$. Hence, in order to characterize such deviations, we introduce the following {\it action functional} $S_{0T}[\phi]$
\begin{equation}
S_{0T}[\phi] = \left\{ \begin{matrix} 
\int_0^T L(\phi,\dot{\phi}) dt, \quad \text {if the integral converges} \\~\\
+\infty, \quad\quad\quad\quad\quad\quad\quad \text {otherwise} 
\end{matrix}\right.  \label{Eq2.6}
\end{equation}  
where $\phi(t) \in \mathcal{C} \bigl([0,\,T], \mathbb{R}^p \bigl)$ and $L(\phi,\dot{\phi})$ is the {\it Lagrangian function} which is given by
\begin{align}
L(\phi,\dot{\phi}) = \frac{1}{2}\left \Vert \dot{\phi}(t) - \sum\nolimits_{k=1}^K \pi_k(\phi(t)) f_k(\phi(t)) \right \Vert^2 \label{Eq2.7}
\end{align}
Then, the probability of observing any sample paths close to a given continuous function $\phi(t) \in \mathcal{C} \bigl([0,\,T], \mathbb{R}^p \bigl)$ can be estimated as follows
\begin{align}
\mathbb{P} \biggl \{  \sup\limits_{ t \in [0,\, T]} \bigl \vert  X_t^{\epsilon} - \phi(t) \bigr \vert < \delta \biggr\}  \asymp \exp\biggl(-\frac{1}{\epsilon} S_{0T}[\phi]\biggr),  \label{Eq2.8}
\end{align}
for sufficiently small $\delta$, where the notation $\asymp$ denotes {\it log-asymptotic equivalent}, i.e, the ratio of the logarithms of both sides converges to one.

In particular, the probability of observing the random event $X_T^{\epsilon} \in \Gamma$ (cf. Assumption~\ref{ASM1}) consists of contributions of sample paths close to all possible absolutely continuous function $\phi(t) \in \mathcal{C}_{\Gamma}$, i.e.,
\begin{align}
\phi(t) \in \mathcal{C}_{\Gamma} = \biggl\{ \phi(t) \in \mathcal{C} \bigl([0,\,T],\, \mathbb{R}^p \bigl) \, \bigg\vert \, \phi(0) = x_0 \in D \cup \partial D ~ \text{and} ~ \phi(T) \in \Gamma \biggr\}.  \label{Eq2.9}
\end{align}
Moreover, in the low-noise limits as $\epsilon \to 0$, the only contribution significantly comes from the trajectory $\phi^{\ast}(t)$ with the {\it smallest action} functional $S_{0T}[\phi^{\ast}]$, i.e.,  
\begin{align}
\phi^{\ast}(t)  = \argmin\limits_{ \phi(t) \in \mathcal{C}} S_{0T}[\phi],  \label{Eq2.10}
\end{align}
which is the most likely (or optimal) sample path. That is, it constitutes {\it almost surely} all sample paths conditioned on the event which is arbitrarily close to $\phi^{\ast}(t)$, while the functional $S_{0T}$ effectively characterizes the difficulty of the passage of $X_t^{\epsilon}$ near $\phi^{\ast}(t)$ in the interval $[0, T]$. More precisely, for sufficiently small $\delta >0$, we have the following result
\begin{align}
\lim\limits_{\epsilon \to 0} \mathbb{P} \biggl \{ \sup\limits_{ t \in [0,\, T]} \bigl \vert  X_t^{\epsilon} - \phi^{\ast}(t) \bigr \vert < \delta \,\biggl\vert \,X_T^{\epsilon} \in \Gamma \biggr\} = 1.  \label{Eq2.11}
\end{align}

Moreover, from calculus of variations (e.g., see \cite{r10}), the optimal sample path $\phi^{\ast}(t)$ can be obtained as a solution to the above variational problem satisfying a necessary condition of extremum $\delta S_{0T}[\phi] = 0$, i.e.,
\begin{align}
\delta S_{0T}[\phi] &= S_{0T}[\phi + \delta\phi] - S_{0T}[\phi], \notag \\
                              &= \int_0^T L(\phi + \delta\phi,\dot{\phi}+\delta\dot{\phi})dt - \int_0^T L(\phi,\dot{\phi})dt, \notag \\
                              &= \int_0^T \left(\frac{\partial L(\phi,\dot{\phi})}{\partial \phi} - \frac{d}{dt} \frac{\partial L(\phi,\dot{\phi})}{\partial \dot{\phi}} \right) \delta \phi dt, \notag\\
                              &= 0. \label{Eq2.12}
\end{align}
Since the variation $\delta \phi$ is arbitrary, then we are left with the Euler-Lagrange equation
\begin{align}
\frac{\partial L(\phi,\dot{\phi})}{\partial \phi} - \frac{d}{dt} \frac{\partial L(\phi,\dot{\phi})}{\partial \dot{\phi}} = 0,  \label{Eq2.13}
\end{align}
which can be solved for the optimal sample path $\phi^{\ast}(t)$ from the following second-order differential equation, with appropriate boundary conditions,
\begin{align}
\ddot{\phi}(t) - \left[ \sum\nolimits_{k=1}^K \pi_k(\phi) f_k(\phi)\right] \left[ \sum\nolimits_{k=1}^K \pi_k(\phi) f_k(\phi)\right]^T  = 0. \label{Eq2.14}
\end{align}
\begin{remark}
Here, it is worth remarking that the optimal sample path $\phi^{\ast}(t)$ can be found by transforming the original stochastic problem into a deterministic problem described by the Euler-Lagrange equation in Equation~\eqref{Eq2.13}.
\end{remark}

On the other hand, the minimization problem, i.e, the variational problem, in Equation~\eqref{Eq2.10} can be solved by seeking solutions to the Euler-Lagrange equation via the Hamiltonian principle from classical mechanics. We first define the conjugate moment as follows
\begin{align}
\psi = \frac{\partial L(\phi,\dot{\phi})}{\partial \dot{\phi}}. \label{Eq2.15}
\end{align}
Then, we can write the Hamiltonian as the Fenchel-Legendre transform of the Lagrangian function $L(\phi, \dot{\phi})$ in $\dot{\phi}$ as follows
\begin{align}
H(\phi,\psi) = \sup\limits_{y} \bigl(\langle \psi,\, y\rangle - L(\phi,y) \bigr),  \label{Eq2.16}
\end{align}
such that the Lagrangian, assuming convexity of $L(\phi, \dot{\phi})$ in $\dot{\phi}$ and noting the duality relation, can be further expressed as
\begin{align}
L(\phi,\dot{\phi}) = \sup\limits_{\psi} \bigl(\langle \psi,\, \dot{\phi} \rangle - H(\phi,\psi) \bigr).  \label{Eq2.17}
\end{align}
Consequently, the minimization problem in Equation~\eqref{Eq2.10} is equivalent to solving the following Hamiltonian equations of motion, with appropriate boundary conditions,
\begin{align}
\dot{\phi} = \nabla_{\psi} H(\phi,\psi) \quad \text{and} \quad \dot{\psi} = -\nabla_{\phi} H(\phi,\psi),  \label{Eq2.18}
\end{align}
where 
\begin{align}
H(\phi,\psi) = \left \langle \sum\nolimits_{k=1}^K \pi_k(\phi) f_k(\phi), \, \psi \right \rangle + \tfrac{1}{2}\left \langle \psi, \psi \right \rangle,  \label{Eq2.19}
\end{align}
that further gives us the following Hamiltonian equations
\begin{align}
\dot{\phi}(t) = \sum\nolimits_{k=1}^K \pi_k(\phi) f_k(\phi) + \psi(t)  \label{Eq2.20}
\end{align}
and
\begin{align}
 \dot{\psi}(t) = -\sum\nolimits_{k=1}^K  \left( f_k(\phi) \left(\nabla \pi_k(\phi)\right)^T +  \pi_k(\phi) \left(\nabla f_k(\phi)\right)^T \right) \psi(t),  \label{Eq2.21}
\end{align}
with $\pi_k(x) > 0$, for $k =1,\,2,\ldots, K$, and ~ $\sum\nolimits_{k=1}^K \pi_k(x) = 1$.

\subsection{Optimization with rare-event modeling} \label{S2.1}
Consider the following random event that we can associate with a certain noise-induced rare-event
\begin{align}
\Phi\bigl(X_T^{\epsilon}\bigr) \le \zeta, \quad \text{with} \quad X_T^{\epsilon} \in \Gamma, \label{Eq2.22}
\end{align}
where $\zeta$ is some positive number and $\Phi\colon \mathbb{R}^p \to \mathbb{R}_{+}$ (e.g., see \cite{r10} for a similar discussion). Then, in the low-noise limits as $\epsilon \to 0$, the probability of observing such a random event $\Phi\bigl(X_T^{\epsilon}\bigr) \le \zeta$, with $X_T^{\epsilon} \in \Gamma$ (cf. Assumption~\ref{ASM1}), subject to $X_0^{\epsilon}=x_0 \in D \cup \partial D$, satisfies the following logarithmic order
\begin{align}
\mathbb{P} \bigl( \Phi\bigl(X_T^{\epsilon}\bigr) \le \zeta \bigr) \asymp \exp\biggl ( -\frac{1}{\epsilon} \inf\limits_{\phi \in \mathcal{C}_{\zeta}} S_{0T}[\phi] \biggr),  \label{Eq2.23}
\end{align}
where $\mathcal{C}_{\zeta} = \bigl\{ \phi(t) \in \mathcal{C}_{\Gamma} \bigl([0,\,T],\, \mathbb{R}^p \bigl) \, \big\vert \, \phi(0) = x_0 \in D \cup \partial D ~ \text{and} ~ \Phi(\phi(T)) \le \zeta \bigr\}$. Here, we remark that the requirement $\Phi(\phi(T)) \le \zeta$, with $\phi(T) \in \Gamma$ (or equivalently, the random event $\Phi\bigl(X_T^{\epsilon}\bigr) \le \zeta$, with $X_T^{\epsilon} \in \Gamma$), can be considered as a desirable outcome that can be included as a penalty term in the minimization problem of Equation~\eqref{Eq2.10} by imposing the following regularity condition on $\Phi(\phi)$ (e.g., see \cite{r11} for a similar discussion).
\begin{assumption}\label{ASM2}
$\nabla\Phi(\phi)$ satisfies the Lipschitz condition in $\phi$.
\end{assumption} 

Noting the condition in Assumption~\ref{ASM2} above, if we define the following rate function
\begin{align}
I(\zeta) = \inf_{\phi \in \mathcal{C}_{\Gamma}} S_{0T}[\phi].  \label{Eq2.24}
\end{align}
Then, the corresponding Fenchel-Legendre transform is given by
\begin{align}
I^{\ast}(\eta) = \inf_{\phi \in \mathcal{C_{\zeta}}} \bigl(S_{0T}[\phi] - \eta \Phi(\zeta) \bigr).  \label{Eq2.25}
\end{align}
Note that, in terms of the variational argument, the corresponding Hamiltonian equations will become as follows
\begin{align}
\text{\em Forward Equation:} ~~ \dot{\phi}(t) = \sum\nolimits_{k=1}^K \pi_k(\phi) f_k(\phi) + \psi(t), \quad \phi(0) = x_0 \quad\quad\quad\quad \quad \label{Eq2.26}
\end{align}
and
\begin{align}
\text{\em Backward Equation:} ~~ \dot{\psi}(t) &= -\sum\nolimits_{k=1}^K  \left( f_k(\phi) \left(\nabla \pi_k(\phi)\right)^T +  \pi_k(\phi) \left(\nabla f_k(\phi)\right)^T \right) \psi(t), \notag \\
& \quad \quad\quad\quad\quad\quad\quad\quad\quad\quad ~~~ \psi(T) = -\lambda \nabla \Phi(\phi(T)), \label{Eq2.27}
\end{align}
where the {\it boundary conditions} for the above forward-backward equations are decoupled (cf. Equations~\eqref{Eq2.20} and \eqref{Eq2.21}). 

In the following section, we exploit such decoupled forward-backward boundary conditions in our rare-event algorithm for solving numerically the variational problem formulation (of the form of Equation~\eqref{Eq2.10}) modeling the most likely sample path leading to a certain noise-induced rare-event.

\section{Main results}\label{S3} In this section, we present our main results. Here, we specifically use the mathematical argument from Section~\ref{S2} and we further provide a new perspective on the learning dynamics for a class of learning problems, i.e., a family of empirical risk minimization-based learning problems arising in modeling of high-dimensional nonlinear functions or dynamical systems.

To be more specific, for an instance of {\it divide-and-conquer} paradigm dealing with large datasets, we first generate $K$ subsample datasets of size $m$ (where $m$ is much less than the total dataset points) from a given original dataset $Z^n = \bigl\{ (x_i, y_i)\bigr\}_{i=1}^n$ by bootstrapping with/without replacement or other related resampling-based techniques, i.e.,
\begin{align}
\hat{Z}^{(1)} = \bigl\{ (\hat{x}_i^{(1)}, \hat{y}_i^{(1)})\bigr\}_{i=1}^m,\, \hat{Z}^{(2)} = \bigl\{ (\hat{x}_i^{(2)}, \hat{y}_i^{(2)})\bigr\}_{i=1}^m, \, \ldots , \, \hat{Z}^{(K)} = \bigl\{ (\hat{x}_i^{(K)}, \hat{y}_i^{(K)})\bigr\}_{i=1}^m.  \label{Eq3.1}
\end{align}
Then, the learning problem of point estimations is to search for parameters $\theta \in \Gamma$, from a finite-dimensional parameter space $\mathbb{R}^p$, such that the function $h_{\theta}(x) \in \mathcal{H}$, from a given class of hypothesis function space $\mathcal{H}$, that describes best the original complete dataset. In terms of mathematical optimization construct, searching for an optimal parameter $\theta^{\ast} \in \Gamma$ can be associated with a {\it steady-state solution} to the following averaged gradient dynamical system, whose {\it time-evolution} is guided by the subsampled datasets from Equation~\eqref{Eq3.1}, i.e.,
\begin{align}
 \dot{\theta}(t) = -\sum\nolimits_{k=1}^K \pi_k(\theta) \nabla J_k(\theta,\hat{Z}^{(k)}), \quad \theta(0) = \theta_0, \label{Eq3.2}
\end{align}
where $\pi_k(\theta) > 0$, for all $k \in \{1,\,2,\ldots, K\}$ and $\theta \in \Gamma$, with $\sum\nolimits_{k=1}^K \pi_k(\theta) = 1$,
\begin{align}
 J_k(\theta, \hat{Z}^{(k)}) = \frac{1}{m} \sum\nolimits_{i=1}^m {\ell}\bigl(h_{\theta}(\hat{x}_i^{(k)}), \hat{y}_i^{(k)}\bigr), \quad k=1,2, \ldots, K \label{Eq3.3}
\end{align}
and $\ell$ is a suitable loss function that quantifies the lack-of-fit between the model.

Moreover, based on the insights from Section~\ref{S2}, we can reinterpret the above averaged gradient dynamical system in Equation~\eqref{Eq3.2}, in the low-noise limits as $\epsilon \to 0$, as a family of randomly perturbed dynamical systems coupled with diffusion-transmutation processes, i.e.,
\begin{align}
d \Theta_t^{\epsilon} &= - \nabla J_{ \nu_k^{\epsilon}}(\Theta,\hat{Z}^{(\nu_k^{\epsilon})}) dt + \sqrt{\epsilon} I_p d W_t, \quad \Theta_0^{\epsilon} = \theta_0, \quad \label{Eq3.4}\\ ~\notag\\
 \mathbb{P} \left \{\nu_{t + \Delta}^{\epsilon} = l \biggl \vert \Theta_t^{\epsilon} = \theta, \nu_{t}^{\epsilon} = k \right\} & = \frac{p_{kl}(\theta)}{\epsilon} \Delta + o({\Delta}) \quad \text{as} \quad\Delta \downarrow 0, \label{Eq3.5}\\
& \quad k, l \in \left\{1,\,2,\ldots, K\right\} \quad \text{and} \quad \theta \in \mathbb{R}^p. \notag
\end{align}
Note that, if Propositions~\ref{P1} and \ref{P2} hold true with respect to the above system in Equations~\eqref{Eq3.4} and \eqref{Eq3.5}. Then, the invariant distribution for the stochastic matrix corresponding to Equation~\eqref{Eq3.5} represents the averaging weighting coefficients that are dynamically apportioning to the learning process corresponding to Equation~\eqref{Eq3.2}, where such an interpretation can be can be viewed as a general framework for improving or enhancing the learning dynamics. Moreover, in order to confine the process $\Theta_T^{\epsilon}$ (i.e., the optimal parameter $\theta^{\ast}$) in the compact set $\Gamma$ with high probability (cf. Assumption~\ref{ASM1}), we require some additional conditions (see Assumption~\ref{ASM3} below) with respect to following averaged dynamical system with small random perturbations
\begin{align*}
 d \Theta_t^{\epsilon} = - \sum\nolimits_{k=1}^K \pi_k(\Theta_t^{\epsilon}) \nabla J_k(\Theta_t^{\epsilon},\hat{Z}^{(k)})dt + \sqrt{\epsilon} I_p d W_t, \quad \Theta_0^{\epsilon} = \theta_0.
\end{align*}
\begin{assumption}\label{ASM3}
Cost functional $J_k(\theta,\hat{Z}^{(k)})$, for each $k \in \{1,\,2, \ldots, K\}$, in Equation~\eqref{Eq3.3} satisfies the following conditions:
\begin{enumerate} [(i)] 
\item {\it Coercivity or superlinear growth condition:}
\begin{align*}
 \lim_{\theta \to \infty} \frac{J_k(\theta, \hat{Z}^{(k)})}{\vert \theta \vert} \to +\infty.
\end{align*}
\item {\it Tightness condition:}
\begin{align*}
\int_{\big\{\theta \colon J_k(\theta, \hat{Z}^{(k)}) \ge \alpha \big\}} \exp \big(-\tfrac{1}{\epsilon} J_k(\theta, \hat{Z}^{(k)}) \big) d \theta \le C_{\alpha} \exp\big(-\alpha/ \epsilon\big), 
\end{align*}
where the constant $C_{\alpha}$ depends on $\alpha$, but not on $\epsilon$.
\end{enumerate}
\end{assumption}
Then, using the mathematical argument related to rare-even modeling from Subsection~\ref{S2.1}, we can provide a large deviation result for a certain noise-induced rare-event. In particular, we consider the following event, e.g., a desirable outcome that may represent how near the solution to an optimal parameter value 
\begin{align}
\Phi\bigl(\Theta_T^{\epsilon}\bigr) \le \zeta, \quad \text{with} \quad \Theta_T^{\epsilon} \in \Gamma. \label{Eq3.6}
\end{align}
Note that, in the low-noise limits as $\epsilon \to 0$, the probability of observing such a random event $\Phi\bigl(\Theta_T^{\epsilon}\bigr) \le \zeta$, with $\Theta_T^{\epsilon} \in \Gamma$, subject to $\Theta_0^{\epsilon}=\theta_0$, satisfies the following
\begin{align}
\limsup_{\epsilon \to 0} -\epsilon \log \mathbb{P} \bigl( \Phi\bigl(\Theta_T^{\epsilon}\bigr) \le \zeta \bigr) = \inf\limits_{\phi \in \mathcal{C}_{\zeta}} S_{0T}[\phi],  \label{Eq3.7}
\end{align}
where $\phi(t) \in \mathcal{C}_{\zeta} = \bigl\{ \phi(t) \in \mathcal{C}_{\Gamma} \bigl([0,\,T],\, \mathbb{R}^p \bigl) \, \big\vert \, \phi(0) = \theta_0, ~ \Phi(\phi(T)) \le \zeta \bigr\}$ and the {\it action functional} $S_{0T}[\phi]$ is given by
\begin{equation}
S_{0T}[\phi] = \left\{ \begin{matrix} 
\int_0^T L(\phi,\dot{\phi}) dt, \quad \text {if the integral converges} \\~\\
+\infty, \quad\quad\quad\quad\quad\quad\quad \text {otherwise} 
\end{matrix}\right.  \label{Eq3.8}
\end{equation}  
while the corresponding {\it Lagrangian function} $L(\phi,\dot{\phi})$ is given by
\begin{align}
L(\phi,\dot{\phi}) = \frac{1}{2}\left \Vert \dot{\phi}(t) + \sum\nolimits_{k=1}^K \pi_k(\phi) \nabla J_k(\phi,\hat{Z}^{(k)}) \right \Vert^2 \label{Eq3.9}
\end{align}
Moreover, if we consider the condition $\Phi(\phi(T)) \le \zeta$, with $\phi(T) \in \Gamma$, as a penalty term in the corresponding variational problem (cf. Equation~\eqref{Eq2.10}). Then, we will have the following Hamiltonian equations, with decoupled boundary conditions,
\begin{align}
 \dot{\phi}(t) = -\sum\nolimits_{k=1}^K \pi_k(\phi) \nabla J_k(\phi,\hat{Z}^{(k)}) + \psi(t), \quad \phi(0) = \theta_0 \quad \label{Eq3.10}
\end{align}
and
\begin{align}
\dot{\psi}(t) = \sum\nolimits_{k=1}^K  \biggl( \nabla J_k(\phi,\hat{Z}^{(k)}) \left(\nabla \pi_k(\phi)\right)^T + \pi_k(\phi) \left(\operatorname{Jac}_{\nabla J}(\phi,\hat{Z}^{(k)}) \right)^T \biggr) \psi(t), \notag \\
 \quad \quad \quad  \psi(T) = -\lambda \nabla \Phi(\phi(T)), \label{Eq3.11}
\end{align}
where $\operatorname{Jac}_{\nabla J}(\phi,\hat{Z}^{(k)})$ is the Jacobian of $\nabla J_k(\phi,\hat{Z}^{(k)})$, with $\pi_k(\phi) > 0$, for $k =1,\,2,\ldots, K$, and $\sum\nolimits_{k=1}^K \pi_k(\phi) = 1$. Note that, if the stochastic matrix in Equation~\eqref{Eq3.5} does not depend on $\phi$, but irreducible and aperiodic. Then, the corresponding invariant distribution $\pi=(\pi_1, \pi_2, \dots, \pi_K)$ is strictly positive and constant, that corresponds to the left eigenvector for the stochastic matrix. Moreover, the rare-event algorithm, i.e., a numerical scheme, consists of the following steps (including resampling technique from the original dataset) for finding the most likely sample path.

{\rm \footnotesize

\begin{framed}

{\bf ALGORITHM:} Rare-Event Algorithm
\begin{itemize}
\item[{\bf Input:}] The original dataset $Z^n = \bigl\{ (x_i, y_i)\bigr\}_{i=1}^n$; $K$ number of subsampled datasets; $m$ subsample data size. Then, using bootstrapping technique, with replacement, generate $K$ subsample datasets:
\begin{align*}
\hat{Z}^{(1)} = \bigl\{ (\hat{x}_i^{(1)}, \hat{y}_i^{(1)})\bigr\}_{i=1}^m,\, \hat{Z}^{(2)} = \bigl\{ (\hat{x}_i^{(2)}, \hat{y}_i^{(2)})\bigr\}_{i=1}^m, \, \ldots , \, \hat{Z}^{(K)} = \bigl\{ (\hat{x}_i^{(K)}, \hat{y}_i^{(K)})\bigr\}_{i=1}^m
\end{align*}
\item[{\bf 0.}] Start with the ${\rm\,j^{th}}$ guess $\phi^j(t) \in \mathcal{C}_{\zeta}$ for the most likely trajectory path.\footnote{Notice that the algorithm works on the space $\mathcal{C}_{\zeta}$ which is more convenient, rather than on $\mathcal{C}_{\Gamma}$.}
\item[{\bf 1.}] Solve the equation
\begin{align*}
\dot{\psi}(t) = \sum\nolimits_{k=1}^K \pi_k \left(\operatorname{Jac}_{\nabla J}(\phi,\hat{Z}^{(k)}) \right)^T \psi(t), \quad \psi(T) = -\lambda \nabla \Phi(\phi^j(T))
 \end{align*}
 backward in time to obtain $\psi^{j}(t)$.
 \item[{\bf 2.}] Solve the equation
\begin{align*}
\dot{\phi}(t) = - \sum\nolimits_{k=1}^K \pi_k \nabla J_k(\phi,\hat{Z}^{(k)}) + \psi^j(t), \quad \phi(0) = \theta_0
\end{align*}
 forward in time to obtain the next guess  $\phi^{j+1}(t)$.
 \item[{\bf 3.}] Iterate until convergence, i.e., $\Vert \phi^{j+1}(t) - \phi^{j}(t)\Vert \le {\rm {\rm tol}}$.
 \item[{\bf Output:}] An optimal parameter value $\theta^{\ast} = \phi^{j}(T)$ as $j \to \infty$.
\end{itemize}
\end{framed}}

\begin{remark}
Note that, in Step 2 above, the functional $\psi^j(t)$ enters as an additional term to the system equation that nudges the dynamics toward the rare-event by reweighing the sample paths which is similar to that of important sampling concept (e.g., see \cite{r11} for additional discussions).
\end{remark}

\section{Numerical simulation results}\label{S4} In this section, we presented numerical results for a simple polynomial interpolation problem modeling the thermophysical properties of saturated water (in liquid state), where the data is taken from (see \cite[p.~1003]{r12}). Here, the mathematical model, which relates the specific heat $c_{p}$ with the temperature $T$ (in Kelvin ${\rm [K]}$, in the range $273.15 \,{\rm K} \le T \le 373.15\, {\rm K}$, is assumed to obey a second-order polynomial function of the form
\begin{align*}
h_{\theta} (T) = \alpha_0 + \alpha_1 T + \alpha_2 T^2, 
\end{align*}
where $\theta = (\alpha_0, \alpha_1, \alpha_2)$ is the model parameter. Then, we presented results for the following two cases:

{\bf Part (a).} For the first case result, we generated $K=6$ subsampled datasets of size $m = 18$ from the original data points of $22$ by bootstrapping with replacement. Here, we also sampled additional dataset of size $9$, denoted by $\hat{Z}^{Test}$, that will be used for model validation. Our interest is to characterize the learning dynamics, based on the asymptotic behavior of the averaged gradient dynamical system with small noises, which is interpreted as random walks on a graph, with six vertices identified as ${\rm S1}$, ${\rm S2}$, ${\rm S3}$, ${\rm S4}$, ${\rm S5}$ and ${\rm S6}$, whose outgoing edges are uniformly chosen at random. Figure~\ref{FG1} shows the random walk on this graph that begins at some vertex, and moves to another vertex involving a process of moving clockwise, counterclockwise, or staying in place, at each time step, with probabilities of $1/3$, i.e.,
\begin{equation}
p_{kl} = \left\{ \begin{matrix} 
1/3, \quad \text {if} \quad l = (k -1) \operatorname{mod} 6 \vspace{1mm}\\ 
1/3, \quad \text {if} \quad l = k \quad \quad\quad\quad \quad~  \vspace{1mm}\\
1/3, \quad \text {if} \quad l = (k +1) \operatorname{mod} 6\,
\end{matrix}\right. \notag
\end{equation}  
Moreover, the corresponding invariant distribution will approach a uniform distribution $\pi = (1/6,\, 1/6,\,1/6,\,1/6,\,1/6,\,1/6)$, after making transitions of $\epsilon^{-1} \ge 10,000$ (where the parameter $\epsilon$ is also related to perturbed noise term). As a result, the learning dynamics can be represented by a randomly perturbed averaged system, i.e., $d \Theta_t^{\epsilon} = -(1/6) \sum\nolimits_{k=1}^6 \nabla J_k(\Theta_t^{\epsilon},\hat{Z}^{(k)})dt + \sqrt{\epsilon} I_p d W_t$.
\begin{figure}[h]
\begin{center}
 \includegraphics[scale=0.325]{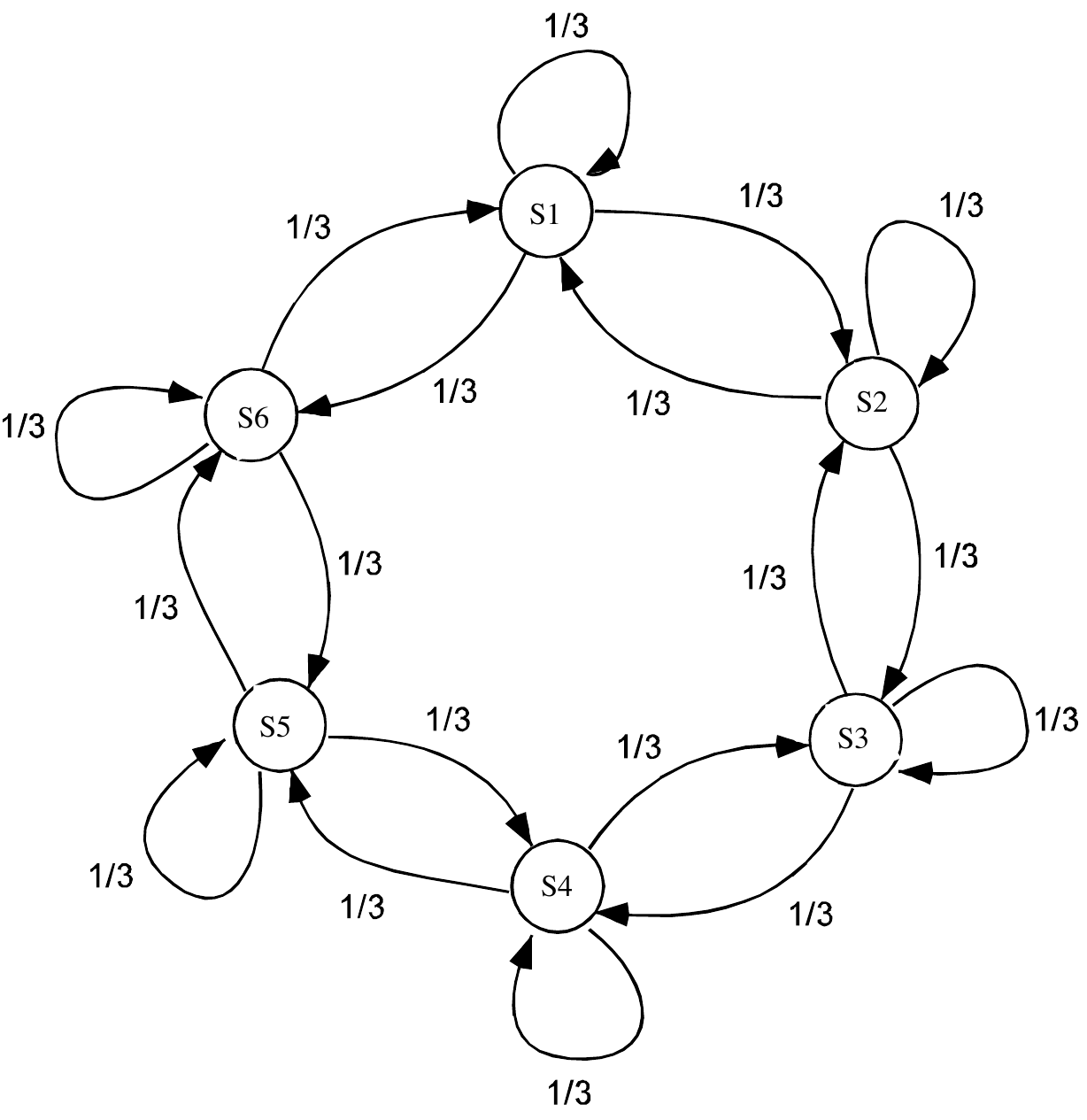}\vspace{-1mm}
  \caption{ A random walk on a graph with six vertices associated with learning dynamics.} \label{FG1}
\end{center}
\end{figure} \vspace{-4mm}
\begin{figure}[bh]
\begin{center}
 \includegraphics[scale=0.215]{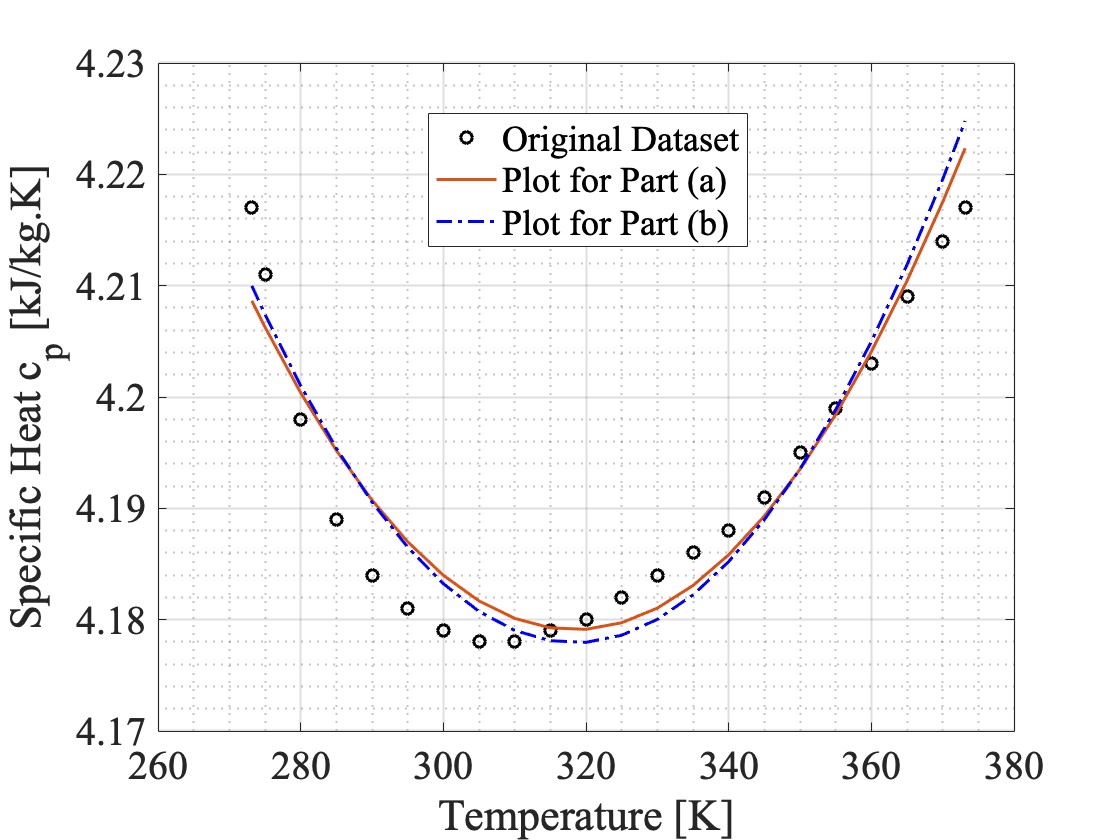}\vspace{-1mm}
  \caption{Plots for the original dataset for the specific heat $c_{\rho}$ and the learned model $h_{\theta^{\ast}}$.} \label{FG2}
\end{center}
\end{figure}

{\bf Part (b).} For the second part, we generated $K=3$ subsampled datasets of size $m = 7$ from the original dataset, by bootstrapping without replacement. Here, the stochastic matrix considered, which is associated with the learning dynamics, is given by
\begin{equation}
P = \left(\begin{matrix} 
0 & 1 & 0 \\
1/3 \ & 0 & 2/3\\
1/3 & 1/3 & 1/3
\end{matrix}\right) \notag
\end{equation} 
Moreover, the corresponding invariant distribution will approach $\pi = (0.25,\, 0.375,\,0.375)$, after making transitions of $\epsilon^{-1} \ge 20$, while the randomly perturbed averaged system can be represented by $d \Theta_t^{\epsilon} = -\sum\nolimits_{k=1}^3 \pi_k \nabla J_k(\Theta_t^{\epsilon},\hat{Z}^{(k)})dt + \sqrt{\epsilon} I_p d W_t$. Note that, for Part (a), with random walks on the graph, we considered an event $\Phi\bigl(\Theta_T^{\epsilon},\hat{Z}^{Test}\bigr) \le \zeta = 1 \times 10^{-4}$ for the rare-event simulation (e.g., see \cite{r11} for related discussions) that signifies model validation at a level of order $1 \times 10^{-4}$, i.e., an event that characterizes how near the solution to an optimal parameters, with steady-state solution of $\Theta_T^{\epsilon} \to \theta^{\ast} = (5.5430,-8.5994 \times 10^{-3},1.3549 \times 10^{-5})$, while for Part (b), a direct computation of the steady-state solution leads to $\Theta_T^{\epsilon} \to \theta^{\ast} = (5.6761,-9.4045 \times 10^{-3},1.4765 \times 10^{-5})$. Figure~\ref{FG2} shows plots for the original dataset for the specific heat $c_p$ and the learned models with the rare-event simulation.

\section{Concluding Remarks}\label{S5}
In this paper, we presented a mathematical argument, based on averaged gradient systems coupled with diffusion-transmutation processes, that provides a new perspective on the learning dynamics for a class of empirical risk minimization-based learning problems. In particular, the framework can be viewed as a general extension of improving the learning dynamics, by apportioning dynamically the averaging weighting coefficients, while the averaged gradient systems are guided by subsampled datasets obtained from the original dataset as an instance of divide-and-conquer paradigm. Here, we also provided a large deviation result, i.e., a variational problem formulation, modeling and a computational algorithm for solving the most likely sample path leading to optimal solutions. Finally, as part of this work, we  presented some numerical results for a typical case of nonlinear regression problem.

\end{document}